\documentclass[12pt]{amsart}

\usepackage{amssymb,mathrsfs,enumitem,amsmath,amsthm,bbold}
\usepackage[mathscr]{euscript}
\usepackage{color}

\usepackage[all]{xy}
\definecolor{Chocolat}{rgb}{0.36, 0.2, 0.09}
\definecolor{BleuTresFonce}{rgb}{0.215, 0.215, 0.36}
\definecolor{EgyptianBlue}{rgb}{0.06, 0.2, 0.65}

\usepackage[OT2,OT1]{fontenc}

\usepackage[backend=bibtex,
    sorting=anyt,
    isbn=false,
    url=false,
    doi=false,
    maxbibnames=100
    ]{biblatex}
\usepackage{fourier}

\usepackage[colorlinks,final,hyperindex]{hyperref}
\hypersetup{citecolor=BleuTresFonce, urlcolor=EgyptianBlue, linkcolor=Chocolat}

\newtheorem{theorem}{Theorem}[section]
\newtheorem{corollary}[theorem]{Corollary}

\newtheorem{lemma}[theorem]{Lemma}
\newtheorem{proposition}[theorem]{Proposition}

\newtheorem{conjecture}[theorem]{Conjecture}

\theoremstyle{definition}

\newtheorem{remark}[theorem]{Remark}

\newcommand{\ac}{\scriptstyle \text{\rm !`}}

\DeclareMathAlphabet{\pazocal}{OMS}{zplm}{m}{n}

\def\calA{\pazocal{A}}
\def\calB{\pazocal{B}}

\def\calF{\pazocal{F}}
\def\calG{\pazocal{G}}

\def\calK{\pazocal{K}}
\def\calL{\pazocal{L}}

\def\calO{\pazocal{O}}

\def\calR{\pazocal{R}}

\def\calW{\pazocal{W}}

\DeclareMathOperator{\Lie}{Lie}

\DeclareMathOperator{\Ass}{Ass}

\DeclareMathOperator{\Alt}{Alt}

\DeclareMathOperator{\Ind}{Ind}

\DeclareMathOperator{\Vect}{{\ensuremath\mathsf{Vect}}}

\DeclareMathAlphabet{\mathbbold}{U}{bbold}{m}{n}

\def\k{\mathbbold{k}}

\begin{document}

\title[On the conjecture of Shang about free alternative algebras]{On the conjecture of Shang\\ about free alternative algebras}

\author{Vladimir Dotsenko}
\address{Institut de Recherche Math\'ematique Avanc\'ee, UMR 7501, Universit\'e de Strasbourg et CNRS, 7 rue Ren\'e-Descartes, 67000 Strasbourg, France}

\email{vdotsenko@unistra.fr}

\begin{abstract}
Kashuba and Mathieu \cite{MR4235202} proposed a conjecture on vanishing of some components of the homology of certain Lie algebras, implying a description of the $GL_d$-module structure of the free $d$-gene\-rated Jordan algebra. Their conjecture relies on a functorial version of the Tits--Kantor--Koecher construction that builds Lie algebras out of Jordan algebras.
Recently, Shang~\cite{shang2025allisonbenkartgaofunctorcyclicityfree} used a functorial construction of Allison, Benkart and Gao that builds Lie algebras out of alternative algebras to propose another conjecture on vanishing of some components of the homology of certain Lie algebras, implying a description of the $GL_d$-module structure of the free $d$-generated alternative algebra. In this note, we explain why the conjecture of Shang is not true. 
\end{abstract}

\maketitle

\section*{Introduction}
An alternative algebra is a nonassociative algebra in which the associator
 \[
(a,b,c)=(ab)c-a(bc)     
 \] 
is a skew-symmetric function of its arguments $a,b,c$. The most famous instance of octonions is given by octonions of Cayley \cite{Cayley01031845}; a systematic theory of alternative algebras was developed by Zorn \cite{Zorn1931}. 
  
A lot of structural results on alternative algebras are known \cite{MR668355}, yet free alternative algebras remain a mysterious object. They have many unexpected properties: for instance, free alternative algebras on sufficiently many generators contain central and nilpotent elements \cite{MR539586}. Recently, Shang \cite{shang2025allisonbenkartgaofunctorcyclicityfree}, inspired by a conjecture of Kashuba and Mathieu on free Jordan algebras \cite{MR4235202}, proposed a conjecture that would shed light on free alternative algebras. Namely, he conjectured that the homology of Lie algebras obtained from free alternative algebras by a construction of Allison, Benkart, and Gao \cite[Th.~4.13]{MR1752782}, if viewed as an $\mathfrak{sl}_3$-module, contains no trivial or adjoint isotypic components in homological degrees greater than one. If true, this conjecture would imply a description of the $GL_d$-module structure of the free $d$-generated alternative algebra for each $d\in\mathbb{N}$. 

We show in this note that the conjecture of Shang is not true. It is worth mentioning that, as indicated to the author by Olivier Mathieu when discussing the results of this note, one heuristic reason that the conjecture of Shang is ``too good to be true'' is that the Allison--Benkart--Gao Lie algebra has three different $\mathfrak{sl}_3$-module types of nontrivial relations, so it is plausible that interaction of those relations creates trivial and adjoint components in higher homological degrees. In fact, we give several different obstructions to the conjecture. Some of them come from the concrete situations where free alternative algebras are well understood: we show that conjecture fails in in degree $17$ for the free alternative algebra on two generators (which coincides with the free associative algebra on two generators thanks to a theorem of Artin \cite{Zorn1931}) and in degree $10$ for the free alternative superalgebra on one odd generator (described explicitly by Shestakov and Zhukavets \cite{MR2355693}), showing that the conjecture fails for all free alternative algebras and superalgebras except for the free alternative algebra on one generator. Moreover, we show that the conjecture predicts the correct dimensions up to degree $6$ for all free alternative algebras and superalgebras, but in degree $7$ it fails for the free alternative algebra on three generators (described explicitly by Iltyakov \cite{MR781229}). All these argument rely on Lemma \ref{lm:residue} that we prove below, which is an analogue of the result of \cite[Sec.~1.11]{MR4235202}. One other obstruction is of more fundamental nature: it turns out that the conjecture predicts a virtual non-effective $S_{10}$-module structure on the component $\Alt(10)$ of the alternative operad. 

Our discovery that the conjecture already fails for the free alternative algebra on two generators is particularly surprising; in fact, it led us to questioning the conjecture of Kashuba and Mathieu in the case of two generators, and to eventually disproving it as well, despite overwhelmingly positive computational evidence in its favour \cite{DH-KM}. (That paper, even though partially written in parallel to the present one, additionally contains new very extensive computational data on free Jordan algebras; the present paper is much more theoretical, even though relies on computer software for verifying a couple of assertions.)

This paper is organized as follows. In Section \ref{sec:conjecture}, we recall the necessary definitions and the main conjecture. In Section \ref{sec:operads}, we discuss an operadic viewpoint of the Allison--Benkart--Gao construction and of the conjecture of Shang and its superalgebra version. In Section \ref{sec:counterexamples}, we explain why the conjecture is false for most free alternative algebras and superalgebras, and indicate the lowest degree in which it breaks. In Section \ref{sec:inner}, we discuss one potentially surviving conjecture on inner derivations of free alternative algebras. Finally, in Appendix~\ref{sec:appendix}, we give the \texttt{SageMath} \cite{sagemath} code that computes the symmetric group actions on the components of the operad of alternative algebras predicted by the conjecture.

\subsection*{Acknowledgements} I am grateful to Frederic Chapoton, Iryna Kashuba, Olivier Mathieu, and Ivan Shestakov for many useful discussions. This work was supported by the French national research agency (project ANR-20-CE40-0016) and by Institut Universitaire de France. Some of the computations presented here were done during the visit to the Shenzhen International Center for Mathematics; I am grateful to this institution for hospitality and excellent working conditions. 

\section{The Allison--Benkart--Gao functor and the conjecture of Shang}\label{sec:conjecture}

Throughout this paper, we use the following notation in any alternative algebra $A$: $f\cdot g=\frac12(fg+gf)$, $L_g(f)=g\cdot f$, $R_g(f)=f\cdot g$, and denote by $\k$ the ground field, which we assume to be of zero characteristic. Recall that for any alternative algebra $A$ and any $a,b\in A$, the operator $D_{a,b}\in \mathrm{End}(A)$ defined by the formula
 \[
D_{a,b}:=[L_a,L_b]+[R_a,R_b]+[L_a,R_b] 
 \]
is a derivation of $A$. These operators satisfy the following identities \cite[Sec.~III.8]{MR210757}:
\begin{gather}
D_{a,b}+D_{b,a}=0,\label{eq:antisym}\\
D_{ab,c}+D_{bc,a}+D_{ca,b}=0,\label{eq:cyclic}\\
[D,D_{a,b}]=D_{D(a),b}+D_{a,D(b)} \text{ for all } D\in \mathrm{Der}(A).
\end{gather} 
Derivations of this form are called \emph{inner}; they arise as a particular case of the general theory of inner derivations of nonassociative algebras \cite[Sec.~II.3]{MR210757}. The last property shows that inner derivations $\mathrm{Inner}(A)$ form an ideal of the Lie algebra $\mathrm{Der}(A)$ of all derivations of $A$. It is known \cite[Sec.~2.35]{MR1752782} that for an alternative algebra $A$, the vector space
$\mathfrak{sl}_3\otimes A\oplus \mathrm{Inner}(A)$ 
has a Lie algebra structure given by
\begin{gather*}
[x\otimes a,y\otimes b]=
[x,y]\otimes a\cdot b+\{x,y\}\otimes[a,b]+\frac13\mathrm{tr}(xy)D_{a,b},\\
[D_{a,b},x\otimes c]=x\otimes D_{a,b}(c),\\
[D_{a,b},D_{c,d}]=D_{D_{a,b}(c),d}+D_{c,D_{a,b}(d)}.
\end{gather*}
(Here and below we denote by $\{x,y\}:=\frac12(xy+yx)-\frac13\mathrm{tr}(xy)I_3$ the standard $\mathfrak{sl}_3$-module projection $S^2\mathfrak{sl}_3\twoheadrightarrow \mathfrak{sl}_3$.) However, inner derivations are not functorial, and hence this construction does not give a functor from the category of alternative algebras to the category of Lie algebras. However, this is remedied if one passes to the universal central extension of this latter algebra, described as follows \cite[Th.~4.13]{MR1752782}. Guided by Equations \eqref{eq:antisym} and \eqref{eq:cyclic}, one associates to the given alternative algebra $A$ the vector space 
 \[
\calB(A):=\Lambda^2(A)/(ab\wedge c+bc\wedge a+ca\wedge b\colon a,b,c\in A)    
 \]
which one can think of as a functorial version of the space of $\mathrm{Inner}(A)$. Furthermore, one defines on the vector space 
 \[
\mathsf{ABG}(A):=  \mathfrak{sl}_3\otimes A\oplus\calB(A)   
 \]
a skew-symmetric bracket by
\begin{gather}
\label{eq:ABG1}
[x\otimes a,y\otimes b]=
[x,y]\otimes a\cdot b+\{x,y\}\otimes[a,b]+\frac{\mathrm{tr}(xy)}3 a\wedge b,\\
[a\wedge b,x\otimes c]=x\otimes D_{a,b}(c),\label{eq:ABG2}\\
[a\wedge b,c\wedge d]=D_{a,b}(c)\wedge d+c\wedge D_{a,b}(d).\label{eq:ABG3}
\end{gather}
One can show that these formulas make $\mathsf{ABG}(A)$ into a Lie algebra that one calls the \emph{Allison--Benkart--Gao Lie algebra} associated to $A$. Clearly, the construction $\mathsf{ABG}(A)$ is functorial, and produces a Lie algebra in the category of $\mathfrak{sl}_3$-modules having only trivial and adjoint components. Shang proves in \cite[Th.~2.2]{shang2025allisonbenkartgaofunctorcyclicityfree} that for each such Lie algebra $\mathfrak{g}$, the multiplicity of the adjoint component has an alternative algebra structure, which he calls the \emph{Berman--Moody functor}. Furthermore, he shows in \cite[Th.~2.3]{shang2025allisonbenkartgaofunctorcyclicityfree} that $\mathsf{ABG}$ is the left adjoint of the Berman--Moody functor. Thus, this situation is parallel to that discussed in \cite{MR4235202}, where the functorial \emph{Tits--Allison--Gao Lie algebra} $\mathsf{TAG}(J)$ built out of a Jordan algebra $J$ is studied; it is a Lie algebra in the category of $\mathfrak{sl}_2$-modules having only trivial and adjoint components. The multiplicity of the adjoint module in such a Lie algebra has a natural Jordan algebra structure; this defines the \emph{Tits functor} from the category of such Lie algebras to the category of Jordan algebras, and $\mathsf{TAG}$ is the left adjoint of the Tits functor, so that in particular $\mathsf{TAG}(\mathrm{Jord}(V))$ is the free Lie algebra the category of $\mathfrak{sl}_2$-modules having only trivial and adjoint components. The conjecture proposed in \cite{MR4235202} asserts that the homology $H_\bullet(\mathsf{TAG}(\mathrm{Jord}(V)),\k)$ has no trivial or adjoint components in homological degrees greater than one; if it is true, it would give character formulas for free Jordan algebras. Shang proposed the following analogue of this conjecture for free alternative algebras.

\begin{conjecture}[{\cite[Conjecture 1]{shang2025allisonbenkartgaofunctorcyclicityfree}}]\label{conj:conj1}
Let $\Alt(V)$ denote the free alternative algebra generated by a finite-dimensional vector space $V$. The $\mathfrak{sl}_3$-module 
 \[
H_k(\mathsf{ABG}(\Alt(V)),\k)     
 \]
has no trivial or adjoint component for $k>1$. 
\end{conjecture}

It is shown in \cite[Lemma 4.1]{shang2025allisonbenkartgaofunctorcyclicityfree} that there exist unique elements $a(V)$ and $b(V)$ in the augmentation ideal of the Grothendieck ring of $GL(V)$ for which we have the following equalities in the Grothendieck ring of $GL(V)\times PSL_3$:
\begin{gather*}
[\lambda(a(V)[L(\alpha_1+\alpha_2)]+b(V)[L(0)])\colon [L(0)]]=[\k]\\
[\lambda(a(V)[L(\alpha_1+\alpha_2)]+b(V)[L(0)])\colon [L(\alpha_1+\alpha_2)]]=-[V].
\end{gather*}
Here $\lambda$ is the homological $\lambda$-operation satisfying $\lambda(x+y)=\lambda(x)\lambda(y)$ and given on each effective class $x=[U]$ by 
 \[
\lambda(x)=\sum_{k\ge 0}(-1)^k[\Lambda^k(U)].     
 \]
The proof of \cite[Lemma 4.1]{shang2025allisonbenkartgaofunctorcyclicityfree} gives an explicit recursive procedure for computing $a(V)$ and $b(V)$; in Appendix \ref{sec:appendix}, we give the \texttt{SageMath} code implementing that procedure.  

According to \cite[Th.~4.1]{shang2025allisonbenkartgaofunctorcyclicityfree}, Conjecture \ref{conj:conj1} implies the following conjecture on character formulas for $\Alt(V)$ and $\calB(\Alt(V))$ in the Grothendieck ring of $GL(V)$. 

\begin{conjecture}[{\cite[Conjecture 2]{shang2025allisonbenkartgaofunctorcyclicityfree}}]\label{conj:conj2}
Let $\Alt(V)$ denote the free alternative algebra generated by a finite-dimensional vector space $V$. In the Grothendieck ring of $GL(V)$, we have
 \[
[\Alt(V)]=a(V), \quad [\calB(\Alt(V))=b(V)].     
 \]
\end{conjecture}

\section{Operads, algebras, and superalgebras}\label{sec:operads}

This section, setting an operad theory context for the conjecture of Shang, is closely related to the corresponding section of \cite{DH-KM}. We start by developing some basics of the theory of operads to assist the readers whose intuition is coming from nonassociative ring theory. We refer to \cite[Sec.~2.3]{MR4675074} for a more general introductory text and to the monograph \cite{MR2954392} for systematic information on operads. Our viewpoint on operads is largely informed by the language of species, or analytic functors, originating in the work of Joyal \cite{MR633783} (referred to as Schur functors in \cite{MR2954392}). Recall that, to a sequence $\{K(n)\}_{n\ge 0}$, where each $K(n)$ is a right module over the symmetric group $S_n$, one can associate a functor $\calK\colon \Vect\to\Vect$ given by the formula
 \[
\calK(V):= \bigoplus_{n\ge 0}K(n)\otimes_{\k S_n}V^{\otimes n} .   
 \]
Functors like that are called \emph{analytic functors}: their value on $V$ is a ``categorified Taylor series'' with the ``iterated derivatives'' $K(n)$ (one can view the tensoring over $\k S_n$ as the categorical division by $n!$). An important class of analytic functors come from free algebras of various kinds. For instance, if we denote by $\Alt(n)$ the subspace of the free alternative algebra $\Alt(x_1,\ldots,x_n)$ consisting of all elements of degree exactly one in each generator $x_1,\ldots,x_n$, this vector space has a natural right $S_n$-action (by permutations of the generators $x_1,\ldots,x_n$), and hence the collection of all these spaces 
 \[
\Alt:=\{\Alt(n)\}_{n\ge 1}
 \]
gives rise to an analytic functor. The term $\Alt(n)\otimes_{\k S_n}V^{\otimes n}$ should be thought of as the result of ``substitution'' of $n$ elements of $V$ into multilinear operations with $n$ arguments that can be defined on alternative algebras; thus, $\Alt(V)$ is naturally identified with the free alternative algebra generated by $V$. Additionally, the natural map 
 \[
\Alt(\Alt(V))\to\Alt(V)     
 \]
which forgets the nested structure of alternative operations on the left, considered together with the obvious inclusion 
$V\hookrightarrow\Alt(V)$ makes the analytic functor $\Alt$ into a monad; monads like that are called \emph{operads}.

Note that over a field of characteristic zero, every identity is equivalent to a multilinear identity, and in particular the alternative identities are equivalent to the multilinear identities
\begin{gather*}
(ab)c-a(bc)+(ba)c-b(ac)=0,\\
(ab)c-a(bc)+(ac)b-a(cb)=0.
\end{gather*}
This allows the reader whose intuition comes from the operad theory to present the operad $\Alt$ by means of generators and relations.

We shall also use the symmetric monoidal category on analytic functors given by the so called \emph{Cauchy product}; if $\calF=\{F(n)\}_{n\ge 0}$ and $\calG=\{G(n)\}_{n\ge 0}$ are two analytic functors, we may define a new analytic functor $\calF\otimes\calG$ whose $n$-th component is given by  
 \[
\bigoplus_{k+l=n}\Ind_{S_k\times S_l}^{S_n}(F(k)\otimes G(l)).     
 \]
It is easy to check that this product is associative, admits symmetry isomorphisms $\calF\otimes\calG\to \calG\otimes\calF$, and has a unit $\mathsf{1}$, and that this data satisfies all axioms of a symmetric monoidal category. In particular, one can talk about Lie algebras in this category, which historically are called \emph{twisted Lie algebras} \cite{MR513566}. Concretely, a twisted Lie algebra $\mathfrak{g}$ may be viewed as a Lie algebra of the form
 \[
\mathfrak{g}=\bigoplus_{n\ge 0}\mathfrak{g}(n)     
 \]
where each $\mathfrak{g}(n)$ is a right $\k S_n$-module, and the Lie bracket maps $\mathfrak{g}(n)\otimes \mathfrak{g}(m)$ to $\mathfrak{g}(n+m)$ and is $S_n\times S_m$-equivariant. Clearly, if $\mathfrak{g}$ is a twisted Lie algebra and $V$ is a vector space, then we may view $\mathfrak{g}$ as an analytic functor and obtain the vector space $\mathfrak{g}(V)$; the twisted Lie algebra structure of $\mathfrak{g}$ induces an honest Lie algebra structure on $\mathfrak{g}(V)$. In particular, any Lie algebra $L$ gives rise to a twisted Lie algebra, if one considers the ``constant'' analytic functor 
\[
\mathsf{1}_L(n)=
\begin{cases}
L, \quad n=0,\\
0, \quad n>0.
\end{cases}     
 \]
Another viewpoint on this same definition is that the analytic functor $\mathfrak{g}$ admits an ``action'' of the analytic functor of the Lie operad, that is, a natural transformation $\Lie\circ\mathfrak{g}\to\mathfrak{g}$ between the composition of functors $\Lie$ and $\mathfrak{g}$ and the functor $\mathfrak{g}$: indeed, a Lie algebra structure on $\mathfrak{g}(V)$ can be viewed as a map $\Lie(\mathfrak{g}(V))\to \mathfrak{g}(V)$ that is consistent with the operad structure of $\Lie$, and this map is natural in $V$. In general, for an operad $\calO$, the notion of a twisted $\calO$-algebra is equivalent to that of a left module over the operad $\calO$. One can also define right modules over an operad, a right $\calO$-module is an analytic functor $\calK$ together with a natural transformation $\calK\circ\calO\to \calK$. Right modules are not algebras, but rather constructions depending on algebras in a way that is stable under algebra endomorphisms, for example, the commutator quotient $A/[A,A]$ of an associative algebra (in the context of algebras satisfying polynomial identities, the notion corresponding to that of an operadic right module is known as a $T$-space).

The following result lifts the Allison--Benkart--Gao construction to the level of twisted Lie algebras, and sheds new light on the functorial properties of that construction. Recall that if $\calR$ is a right $\calO$-module and $\calL$ is a left $\calO$-module, we can form the relative composition $\calR\circ_\calO\calL$ analogous to tensor product of a right and a left module over a ring, see \cite[Sec.~5.1.5]{MR2494775}.

\begin{proposition}\leavevmode
\begin{enumerate}
\item There exists a twisted Lie algebra $\calA\calB\calG$ such that we have a Lie algebra isomorphism 
 \[
 \mathsf{ABG}(\Alt(V))\cong \calA\calB\calG(V).    
 \]
\item The twisted Lie algebra $\calA\calB\calG$ is a right module over the alternative operad. Moreover, for every alternative algebra
$A$, we have a natural isomorphism of left $\Lie$-modules
 \[
\mathsf{1}_{\mathsf{ABG}(A)}\cong \calA\calB\calG\circ_{\Alt} \mathsf{1}_A.     
 \]
\end{enumerate} 
\end{proposition}

\begin{proof}
Let $V$ be a finite-dimensional vector space. The free alternative algebra $\Alt(V)$ can be described as the value of the analytic functor of the alternative operad on the vector space $V$ \cite{MR2954392}:
 \[
\Alt(V)\cong\bigoplus_{n\ge 1}\Alt(n)\otimes_{\k S_n}V^{\otimes n}.     
 \]
Moreover, the functor 
 \[
\calB(\Alt(V))\cong \Lambda^2(\Alt(V))/(ab\wedge c+bc\wedge a+ca\wedge b \colon a,b,c\in \Alt(V))     
 \]
is also analytic. Indeed, we may take the obvious functorial version of the definition of $\calB$, looking, for each $n$, at the elements of $\calB(\Alt(x_1,\ldots,x_n))$ consisting of all elements of degree exactly one in each generator $x_1,\ldots,x_n$. Equivalently, one may define 
 \[
B:=\calB(\Alt)= \Lambda^2(\Alt)/(ab\wedge c+bc\wedge a+ca\wedge b \colon a,b,c\in \Alt) ,   
 \]
where the exterior power is now taken with respect to the Cauchy product. We can now define a twisted Lie algebra structure on the analytic functor
 \[
\calA\calB\calG=\mathfrak{sl}_3\otimes\Alt\oplus \calB(\Alt)     
 \]
by the same formulas \eqref{eq:ABG1}--\eqref{eq:ABG3}. We conclude that 
 \[
\mathsf{ABG}(\Alt(V))=\mathfrak{sl}_3\otimes\Alt(V)\oplus \calB(\Alt(V))\cong (\mathfrak{sl}_3\otimes\Alt\oplus \calB(\Alt))(V)   
 \]
is an analytic functor and a twisted Lie algebra (one can even view it as the Allison--Benkart--Gao construction applied to the twisted alternative algebra $\Alt$), and that $\mathsf{ABG}(\Alt(V))\cong \calA\calB\calG(V)$. 

Furthermore, the operad $\Alt$ is tautologically a right $\Alt$-module, and $\calB(\Alt)$ is a right $\Alt$-module since $\calB(\Alt(V)))$ is manifestly stable under all endomorphisms of $\Alt(V)$, so the analytic functor $\calA\calB\calG$ is a right $\Alt$-module. Moreover,  the Lie algebra structure on $\mathsf{ABG}(\Alt(V))$ clearly commutes with all endomorphisms of $\Alt(V)$, so the twisted Lie algebra structure on $\calA\calB\calG$ commutes with the right $\Alt$-module action.. It remains to notice that
\begin{gather*}
\Alt \circ_{\Alt} \mathsf{1}_A\cong \mathsf{1}_A,\\
\calB(\Alt)\circ_{\Alt} \mathsf{1}_A\cong \mathsf{1}_{\calB(A)},
\end{gather*}
where the first isomorphism is obvious, and the second follows from the fact that $\calB(\Alt)$ is defined by the same formula as $\calB(A)$, and the relative composition product ``computes'' all the products in $A$ after the evaluation of the analytic functor $\calB(\Alt)(A)$ by taking the corresponding coequalizer. Since the twisted Lie algebra structure on $\calA\calB\calG$ commutes with the right $\Alt$-module action, the induced Lie algebra structure on the relative composite product matches the Allison--Benkart--Gao Lie algebra structure on the vector space $\mathsf{ABG}(A)$.
\end{proof}

Let us also outline a functorial viewpoint on Lie algebra homology; for simplicity, we shall focus on the homology with trivial coefficients. In classical textbooks, one would often find the either the definition of the homology as a derived functor
 \[
H_\bullet(L,\k):=\mathrm{Tor}_\bullet^{U(L)}(\k,\k),     
 \]
or a definition via the Chevalley--Eilenberg complex  
 \[
H_\bullet(L,\k):=H_\bullet(\Lambda(L),d),     
 \]
where in $\Lambda(L)$ we place $\Lambda^k(L)$ in the homological degree $k$, and the differential $d$ is given by the formula
 \[
d(x_1\wedge\cdots\wedge x_k)=\sum_{1\le i< j\le k}(-1)^{i+j-1}[x_i,x_j]\wedge x_1\wedge\cdots\wedge \hat{x}_{i}\cdots\wedge \hat{x}_{j}\wedge\cdots\wedge x_k,     
 \] 
where the notation $\hat{x}_i$ means that this factor must be omitted. While these definitions are completely unambiguous, they disguise one conceptual aspect of the definition. The operad of Lie algebras is Koszul, and its Koszul dual is the operad of commutative associative algebras. This allows us to interpret the homology of a Lie algebra $L$ as the homology of its \emph{bar construction} \cite{MR2954392}. In our particular case, the space of exterior forms $\Lambda(L)$ is really a disguise of $S^c(sL)$, the cofree conilpotent cocommutative coassociative coalgebra on $sL$ (the vector space $L$ homologically shifted by one using an odd element $s$), and $d$ is the unique coderivation of $S^c(sL)$ extending the map $S^c(sL)\twoheadrightarrow S^2(sL)\to sL$ made of the projection onto $S^c(sL)\twoheadrightarrow S^2(sL)$ and the map $S^2(sL)\to sL$ corresponding to the Lie bracket (a skew-symmetric bilinear map $V\times V\to V$ is the same as symmetric bilinear map $sV\times sV\to sV$); note that, similarly how derivations of free algebras are determined by restriction to generators, coderivations of cofree conilpotent coalgebras are determined by corestriction to cogenerators. With that in mind, we obtain the following result.

\begin{proposition}\label{prop:Kunneth}
We have 
 \[
H_\bullet(\mathsf{ABG}(V),\k)\cong  H_\bullet(\calA\calB\calG,\k)(V).    
 \]
\end{proposition}

\begin{proof}
The underlying vector space of the Chevalley--Eilenberg complex of the Lie algebra $\mathsf{ABG}(V)$ is, as we already indicated, 
 \[
(S^c(s\mathsf{ABG}(V)),d),   
 \]
and, if we note that $S^c(sV)\cong s\Lie^{\ac}(V)$, where $\Lie^{\ac}$ is the Koszul dual cooperad of the Lie operad \cite[Sec.~7.2]{MR2954392}, the differential $d$ comes from the Koszul twisting cochain $\kappa\colon\Lie^{\ac}\to\Lie$. Thus, computing the differential (and its homology) commutes with the evaluation of analytic functors on~$V$.
\end{proof}

\begin{corollary}\label{cor:SchurWeyl}
Conjecture \ref{conj:conj1} holds for all choices of the vector space of generators $V$ if and only if the $\mathfrak{sl}_3$-module 
 \[
H_k(\calA\calB\calG,\k) 
 \]
has no trivial or adjoint component for $k>1$. 
\end{corollary}

\begin{proof}
Proposition \ref{prop:Kunneth} implies the ``if'' part of the statement. The ``only if'' part, that is the assertion that vanishing of the trivial and the adjoint component of $H_k(\mathsf{ABG}(V),\k)$ for all $V$ implies the same for $H_k(\calA\calB\calG,\k)$, follows from the fact that the ``Taylor coefficients'' of an analytic functor can be uniquely reconstructed from its values using the Schur--Weyl duality \cite{MR1153249}.
\end{proof}

Everything discussed above admits an analogue for free alternative superalgebras: indeed, one may generalize the Allison--Benkart--Gao construction to alternative superalgebras and then simply replace the Lie algebra homology by the Lie superalgebra homology in the statement of Conjecture \ref{conj:conj1}. In fact, such generalizations cost nothing: over the years, the author of this note has been advertising the viewpoint that superalgebras over a given operad are merely algebras over the same operad in a larger symmetric monoidal category \cite{DotRev}, and hence various results that can be stated and proved in terms of the corresponding operad (e.g., the Poincaré--Birkhoff--Witt type theorems \cite{MR4300233} or the Nielsen--Schreier property \cite{MR4675074}) are automatically true for superalgebras if already proved for algebras. To give a yet another example of this approach, let us record the following result.

\begin{proposition}\label{prop:super}
Suppose that Conjecture \ref{conj:conj1} is true for all free alternative algebras. Then it is true for all free alternative superalgebras. 
\end{proposition}

\begin{proof}
The free alternative superalgebra $\Alt(V_0\mid V_1)$ generated by the superspace $V_0\mid V_1$ can also be written as a value of the same analytic functor:
 \[
\Alt(V_0\mid V_1)\cong  \Alt(n)\otimes_{\k S_n}(V_0\mid V_1)^{\otimes n},   
 \]
where the latter formula uses the symmetric monoidal structure on the category of $\mathbb{Z}/2\mathbb{Z}$-graded vector spaces given by the Koszul sign rule 
 \[
\sigma(u\otimes v)=(-1)^{|u||v|}v\otimes u.     
 \]
The Lie superalgebra homology also admits a description as the homology of $(S^c(sL),d)$, where $S^c(sL)$ now denotes the cofree conilpotent cocommutative coassociative co-superalgebra generated by $sL$ (which is nothing but the cofree conilpotent cocommutative coassociative \emph{coalgebra} in the symmetric monoidal category of $\mathbb{Z}/2\mathbb{Z}$-graded vector spaces), and $d$ is the unique coderivation of $S^c(sL)$ extending the map 
 \[
S^c(sL)\twoheadrightarrow S^2(sL)\to sL    
 \]
made of the projection onto $S^c(sL)\twoheadrightarrow S^2(sL)$ and the map $S^2(sL)\to sL$ corresponding to the Lie superalgebra structure. 
The obvious analogue of Proposition~\ref{prop:Kunneth} holds, so 
 \[
H_\bullet(\mathsf{ABG}(\Alt(V_0\mid V_1)),\k)\cong H_\bullet(\calA\calB\calG,\k)(V_0\mid V_1),     
 \]
and the claim follows from Corollary \ref{cor:SchurWeyl}. 
\end{proof}

\section{Counterexamples to the conjecture}\label{sec:counterexamples}

\subsection{The conjecture fails for most free alternative superalgebras}

Our precise statement justifying the failure of the conjecture of Shang is the following result.

\begin{theorem}
Conjecture \ref{conj:conj1} fails for the free alternative superalgebra 
 \[
\Alt(x_1,\ldots,x_m\mid y_1,\ldots,y_n)   
 \]
whenever $m>1$ or $n>0$. 
\end{theorem}

\begin{proof}
The free alternative superalgebra $\Alt(x_1,\ldots,x_m\mid y_1,\ldots,y_n)$ has a grading by the monoid $\mathbb{N}^{m+n}$, and this grading induces a grading on the Lie superalgebra 
 \[
\mathrm{ABG}(\Alt(x_1,\ldots,x_m\mid y_1,\ldots,y_n)),     
 \]
as well as on its homology. Considering the graded components supported on $\mathbb{N}^I\subset \mathbb{N}^{m+n}$ for various subsets $I$ of the set of variables, we see that our assertion would follow if we prove it for $(m\mid n)=(2\mid 0)$ and for $(m\mid n)=(0\mid 1)$. 
The first of these results is proved in Proposition \ref{prop:Alt20} below, and the second in Proposition \ref{prop:Alt01} below.
\end{proof}

To explain the proof in the case $(m\mid n)=(2\mid 0)$, we first prove an analogue of the result of \cite[Sec.~1.11]{MR4235202} in our context.

\begin{lemma}\label{lm:residue}
Conjecture \ref{conj:conj2} implies that the dimensions 
 \[
a_n(p):=\dim\Alt(a_1,\ldots,a_p)_k     
 \]
form the unique sequence satisfying
\begin{multline*}
\mathop{\mathrm{Res}}_{(q_1,q_2)=(0,0)}(q_1+pzq_1^{-1}q_2^{-1})(1-q_1q_2^{-1})(1-q_1q_2^2)(1-q_1^2q_2)\\ \times \prod_{n\ge 1}\bigl[\left(1-z^n(q_1^2q_2+q_1^{-2}q_2^{-1})+z^{2n}\right)\left(1-z^n(q_1q_2^2+q_1^{-1}q_2^{-2})+z^{2n}\right)\bigr.\\ \bigl.\times\left(1-z^n(q_1q_2^{-1}+q_1^{-1}q_2)+z^{2n}\right)\bigr]^{a_n(p)}
\,dq_1dq_2=0.     
\end{multline*}
Here $\mathop{\mathrm{Res}}\limits_{(q_1,q_2)=(0,0)}$, the residue at $(0,0)$, means the coefficient of $q_1^{-1}q_2^{-1}$.
\end{lemma}

\begin{proof}
Let us denote $b_n(p):=\dim\calB(\Alt(a_1,\ldots,a_p))_k$. Since the character of the adjoint representation of $\mathfrak{sl}_3$ is equal to
 \[
q_1^2q_2+q_1q_2^{-1}+q_1^{-1}q_2+q_1q_2^2+q_1^{-2}q_2^{-1}+q_1^{-1}q_2^{-2}+2,     
 \]
the usual character formulas for exterior powers imply that the degree-by-degree generating function of $\mathfrak{sl}_3$-characters of 
 \[
\Lambda(\Alt(a_1,\ldots,a_p)\otimes\mathfrak{sl}_3+\calB(\Alt(a_1,\ldots,a_p)))    
 \]
is equal to the expression
\begin{multline*}
\prod_{n\ge 1}\left[\left(1-z^n(q_1^2q_2+q_1^{-2}q_2^{-1})+z^{2n}\right)\left(1-z^n(q_1q_2^2+q_1^{-1}q_2^{-2})+z^{2n}\right)\right.\\ \left.\times\left(1-z^n(q_1q_2^{-1}+q_1^{-1}q_2)+z^{2n}\right)\right]^{a_n(p)}\prod_{n\ge 1}(1-z^n)^{2a_n(p)+b_n(p)},
\end{multline*}
which we shall denote $\Phi(q_1,q_2,z)$.
We may extract the trivial and the adjoint component using the Weyl character formula \cite{MR1153249}, so that Conjecture~\ref{conj:conj2} implies
\begin{gather*}
\mathop{\mathrm{Res}}_{(q_1,q_2)=(0,0)}q_1^{-1}q_2^{-1}(1-q_1q_2^{-1})(1-q_1^2q_2)(1-q_1q_2^2)\Phi(q_1,q_2,z)=1,\\
\mathop{\mathrm{Res}}_{(q_1,q_2)=(0,0)}q_1(1-q_1q_2^{-1})(1-q_1^2q_2)(1-q_1q_2^2)\Phi(q_1,q_2,z)=-pz.
\end{gather*}
Combining the two with coefficients $pz$ and $1$ and using the $\mathbb{Z}[[z]]$-linearity, we get
 \[
\mathop{\mathrm{Res}}_{(q_1,q_2)=(0,0)}(q_1+pzq_1^{-1}q_2^{-1})(1-q_1q_2^{-1})(1-q_1^2q_2)(1-q_1q_2^2)\Phi(q_1,q_2,z) =0,    
 \]
so it remains to drop the factor $\prod_{n\ge 1}(1-z^n)^{2a_n(p)+b_n(p)}$ using the $\mathbb{Z}[[z]]$-linearity once again, to obtain the desired statement. Uniqueness of the sequence $\{a_n(p)\}$ follows from the fact that the coefficient of $z^n$ in
\begin{multline*}
(q_1+pzq_1^{-1}q_2^{-1})(1-q_1q_2^{-1})(1-q_1q_2^2)(1-q_1^2q_2)\\ \times \prod_{n\ge 1}\left((1-z^n(q_1^2q_2+q_1^{-2}q_2^{-1})+z^{2n})(1-z^n(q_1q_2^2+q_1^{-1}q_2^{-2})+z^{2n})\right.\\ \left.\times(1-z^n(q_1q_2^{-1}+q_1^{-1}q_2)+z^{2n})\right)^{a_n(p)}    
\end{multline*}
can be easily seen to contain $a_n(p)q_1^{-1}q_2^{-1}$ with a nonzero constant coefficient, and so the residue condition gives a recurrence relation from which one may determine all terms of the sequence.

\end{proof}

\begin{proposition}\label{prop:Alt20} 
Conjecture \ref{conj:conj1} fails for the algebra $\Alt(x_1,x_2)$.
\end{proposition}

\begin{proof}
Recall that a theorem of Artin~\cite{Zorn1931} implies that $\Alt(x_1,x_2)$ coincides with free associative algebra on two generators. Thus, $a_n(2)=2^n$ for all $n$. We computed in \texttt{SageMath} \cite{sagemath} the coefficients of $z^k$ in 
\begin{multline}\label{eq:residue}
(q_1+2zq_1^{-1}q_2^{-1})(1-q_1q_2^{-1})(1-q_1q_2^2)(1-q_1^2q_2)\\ \times \prod_{n=1}^{17}\left((1-z^n(q_1^2q_2+q_1^{-2}q_2^{-1})+z^{2n})(1-z^n(q_1q_2^2+q_1^{-1}q_2^{-2})+z^{2n})\right.\\ \left.\times(1-z^n(q_1q_2^{-1}+q_1^{-1}q_2)+z^{2n})\right)^{2^n}
\end{multline}
for $k\le 17$; if one does calculations in $\mathbb{Q}[q_1,q_2,q_1^{-1},q_2^{-1}][z]/(z^{18})$, these computations can be done quite efficiently (and can even be verified by an extremely dedicated human being). It turns out that the coefficients of $z^k$ in \eqref{eq:residue} have the residue $0$ for $k\le 16$, but the coefficient for $k=17$ has the residue equal to $2$. This implies that the prediction for $\dim\Alt(x_1,x_2)_k$ is correct for $k\le 16$, but fails for $k=17$. Thus, Conjecture \ref{conj:conj2} does not hold, and hence Conjecture \ref{conj:conj1} does not hold either.   
\end{proof}

Let us now explain the proof in the case $(m\mid n)=(0\mid 1)$.


\begin{proposition}\label{prop:Alt01}
Conjecture \ref{conj:conj1} fails for the superalgebra $\Alt(\varnothing \mid y_1)$.
\end{proposition}

\begin{proof}
According to the results of Shestakov and Zhukavets who described a basis of $\Alt(\varnothing \mid y_1)$, the dimension $d_n$ of the degree $n$ component of $\Alt(\varnothing \mid y_1)$ is described as follows \cite[Cor.~5.2]{MR2355693}:
 \[
d_1=1,\quad d_2=1,\quad d_3=2,\quad d_n=2(n-3)+\frac12(1+(-1)^{n(n+1)/2})  \text{ for } n>3.   
 \]
Thus, the first few numbers in the sequence $\{d_n\}_{n\ge 1}$ are
 \[
1,1,2,3,4,6,9,11,12,14,\ldots
 \]
Completely analogously to Lemma \ref{lm:residue} (but using the ``superdimensions'' $\mathrm{sdim}(V_0\mid V_1)=\dim(V_0)-\dim(V_1)$), one may prove that if Conjecture \ref{conj:conj2} were true, the dimensions $d_n$ would form the unique sequence satisfying
\begin{multline*}
\mathop{\mathrm{Res}}_{(q_1,q_2)=(0,0)}(q_1-zq_1^{-1}q_2^{-1})(1-q_1q_2^{-1})(1-q_1q_2^2)(1-q_1^2q_2)\\ \times \prod_{n\ge 1}\bigl[\left(1-z^n(q_1^2q_2+q_1^{-2}q_2^{-1})+z^{2n}\right)\left(1-z^n(q_1q_2^2+q_1^{-1}q_2^{-2})+z^{2n}\right)\bigr.\\ \bigl.\times\left(1-z^n(q_1q_2^{-1}+q_1^{-1}q_2)+z^{2n}\right)\bigr]^{(-1)^nd_n}
\,dq_1dq_2=0.     
\end{multline*}
We computed in \texttt{SageMath} \cite{sagemath} the coefficients of $z^k$ in 
\begin{multline}\label{eq:residue}
(q_1-zq_1^{-1}q_2^{-1})(1-q_1q_2^{-1})(1-q_1q_2^2)(1-q_1^2q_2)\\ \times \prod_{n=1}^{17}\left((1-z^n(q_1^2q_2+q_1^{-2}q_2^{-1})+z^{2n})(1-z^n(q_1q_2^2+q_1^{-1}q_2^{-2})+z^{2n})\right.\\ \left.\times(1-z^n(q_1q_2^{-1}+q_1^{-1}q_2)+z^{2n})\right)^{(-1)^nd_n}
\end{multline}
for $k\le 10$; if one does calculations in $\mathbb{Q}[q_1,q_2,q_1^{-1},q_2^{-1}][z]/(z^{11})$, these computations can be done efficiently (and can be verified by a dedicated human being). It turns out that the coefficients of $z^k$ in \eqref{eq:residue} have the residue $0$ for $k\le 9$, but the coefficient for $k=10$ has the residue equal to $-2$. This implies that the prediction for $\dim\Alt(\varnothing \mid y_1)_k$ is correct for $k\le 9$, but fails for $k=10$. This disproves Conjecture \ref{conj:conj2}, and, due to \cite[Th.~4.1]{shang2025allisonbenkartgaofunctorcyclicityfree}, Conjecture \ref{conj:conj1} as well.
\end{proof}

\subsection{The lowest degree where the conjecture fails}

Let us begin with recording the following positive partial result.

\begin{proposition}\label{prop:lowdeg}
Conjecture \ref{conj:conj2} is true for all alternative algebras and superalgebras in degrees not exceeding $6$.
\end{proposition}

\begin{proof}
The prediction of Conjecture \ref{conj:conj2} about the $S_n$-module structure of $\Alt(n)$ for $n\le 6$ is as follows (if one denotes by $V_\lambda$, for $\lambda$ a partition of $n$, the irreducible $S_n$-module corresponding to $\lambda$):
\begin{gather*}
\Alt(1)\cong V_1,\\
\Alt(2)\cong V_{1^2}\oplus V_2,\\
\Alt(3)\cong V_{1^3}^2\oplus V_{2,1}^2\oplus V_3,\\
\Alt(4)\cong V_{1^4}^3\oplus V_{2,1^2}^5\oplus V_{2^2}^2\oplus V_{3,1}^3\oplus V_4,\\ 
\Alt(5)\cong V_{1^5}^4\oplus V_{2,1^3}^{10}\oplus V_{2^2,1}^7\oplus V_{3,1^2}^9\oplus V_{3,2}^5\oplus V_{4,1}^4\oplus V_5,\\ 
\Alt(6)\cong V_{1^6}^6\oplus V_{2,1^4}^{16}\oplus V_{2^2,1^2}^{18}\oplus V_{2^3}^8\oplus V_{3,1^3}^{20}\oplus V_{3,2,1}^{20}\oplus V_{3^2}^5\oplus V_{4,1^2}^{14}\oplus V_{4,2}^9\oplus V_{5,1}^5\oplus V_6. 
\end{gather*}
Using the \texttt{albert} program \cite{10.1145/190347.190358}, one may determine dimensions of multihomogeneous components of low degrees for alternative algebras with at most $6$ generators, which gives the characters of the $GL(V)$-module $\Alt(V)_n$ for $n\le 6$, and hence, via the Schur--Weyl duality \cite{MR1153249}, the $S_n$-module structure of $\Alt(n)$ for such~$n$. We verified that those module structures are given by the predictions above. It remains to note that for any vector space $V$, we have
 \[
\Alt(V)_n=\Alt(n)\otimes_{\k S_n} V^{\otimes n},     
 \]
proving the claim.
\end{proof}

We shall now see that Proposition \ref{prop:lowdeg} is the best possible positive result one can hope for. Concretely, we shall show that in the case $(m\mid n)=(3\mid 0)$, the conjecture fails already in degree $7$. For that, we recall a result of Iltyakov \cite{MR781229} who described a basis of $\Alt(x_1,x_2,x_3)$. For brevity, we denote below $L_{f,g}(h)=R_gR_f-R_{f\cdot g}$.

\begin{theorem}[{\cite{MR781229}}]\label{th:Ilt}
Consider the following elements in $\Alt(x_1,x_2,x_3)$:
\begin{gather*}
\mathrm{W}_0:=\{L_{x_1,x_2}^{n_1}L_{x_2,x_3}^{n_2}L_{x_3,x_1}^{n_3}L_{x_1,x_1}^{n_4}L_{x_2,x_2}^{n_5}L_{x_3,x_3}^{n_6}L_{x_1,[x_2,x_3]}^{n_7}(x_1,x_2,x_3)\colon n_1,\ldots,n_7\ge 0\},\\
\mathrm{W}_1:=\{[w,x_1],[w,x_2],[w,x_3]\colon w\in W_0\},\\
\mathrm{W}_2:=\{(w,x_1,x_2),(w,x_2,x_3),(w,x_1,x_3)\colon w\in W_0\},\\
\mathrm{W}:=\calW_0\sqcup \calW_1\sqcup \calW_2,\\
\mathrm{W}':=\{w\cdot (x_1,x_2,x_3)\colon w\in \calW_0\},\\
\mathrm{B}:=\{R_{x_3}^{n_3}R_{x_2}^{n_2}R_{x_1}^{n_1}w\colon n_1,n_2,n_3\ge 0, w\in \calW\sqcup\calW'\}
\end{gather*}
Then we have a vector space isomorphism 
 \[
\Alt(x_1,x_2,x_3)\cong\Ass(x_1,x_2,x_3)\oplus\k\mathrm{B}.     
 \]
\end{theorem}

This allows us to disprove Conjecture \ref{conj:conj1} in the case $(m\mid n)=(3\mid 0)$.

\begin{proposition}\label{prop:Alt30}
Conjecture \ref{conj:conj1} fails for the algebra $\Alt(x_1,x_2,x_3)$.
\end{proposition}

\begin{proof}
The vector space isomorphism  
 \[
\Alt(x_1,x_2,x_3)\cong\Ass(x_1,x_2,x_3)\oplus\k\mathrm{B}     
 \]
of Theorem \ref{th:Ilt} can be used to prove the following formula for the generating function of dimensions of homogeneous components of $\Alt(x_1,x_2,x_3)$:
\begin{equation}\label{eq:IltChar}
\frac{1}{1-3t}+\frac{t^3}{(1-t)^6(1-t^2)^3(1-t^3)}.     
\end{equation}
Indeed, we first note that the vector space $\k\mathrm{W}_0$ looks like the module with one generator of weight $3$ over the algebra of polynomials in six generators of weight $2$ and one generator of weight $3$, so the corresponding generating function is
 \[
\frac{t^3}{(1-t^2)^6(1-t^3)}.   
 \]
The vector space $\k\mathrm{W}_1$ is the direct sum of three copies of $\k\mathrm{W}_0$ with weight shifted by one, so the corresponding generating function is
 \[
\frac{3t^4}{(1-t^2)^6(1-t^3)}.   
 \]
The vector space $\k\mathrm{W}_2$ is the direct sum of three copies of $\k\mathrm{W}_0$ with weight shifted by two, so the corresponding generating function is
 \[
\frac{3t^5}{(1-t^2)^6(1-t^3)}.   
 \]
The vector space $\k\calW'$ is one copy of $\k\calW_0$ with weight shifted by three, so the corresponding generating function is
 \[
\frac{t^6}{(1-t^2)^6(1-t^3)}.   
 \] 
Therefore, the vector space $\k\mathrm{W}\oplus\k\mathrm{W}'=\k\mathrm{W}_0\oplus\k\mathrm{W}_1\oplus\k\mathrm{W}_2\oplus\k\mathrm{W}'$ has the generating function of dimensions of homogeneous components
 \[
\frac{t^3(1+3t+3t^2+t^3)}{(1-t^2)^6(1-t^3)}=\frac{t^3(1+t)^3}{(1-t^2)^6(1-t^3)}=\frac{t^3}{(1-t)^3(1-t^2)^3(1-t^3)}.   
 \] 
Finally, the vector space $\k\mathrm{B}$ looks like the module generated by the vector space $\k\mathrm{W}\oplus\k\mathrm{W}'$ over the algebra of polynomials in three variables of weight $1$, so the corresponding generating function is
 \[
\frac{t^3}{(1-t)^6(1-t^2)^3(1-t^3)} .    
 \]
To complete the proof of Formula \eqref{eq:IltChar}, it remains to recall that the generating function of dimensions of homogeneous components of $\Ass(x_1,x_2,x_3)$ is $\frac{1}{1-3t}$. 

Expanding Formula \eqref{eq:IltChar} to order $7$, we obtain
 \[
1 + 3t + 9t^2 + 28t^3 + 87t^4 + 267t^5 + 804t^6 + 2388t^7 + O(t^8).     
 \]
We may once again use Lemma \ref{lm:residue}. Like in the proof of Proposition \ref{prop:Alt20}, we computed in \texttt{SageMath} \cite{sagemath} the coefficients of $z^k$ in 
\begin{multline}\label{eq:residue}
(q_1+2zq_1^{-1}q_2^{-1})(1-q_1q_2^{-1})(1-q_1q_2^2)(1-q_1^2q_2)\\ \times \prod_{n=1}^{7}\left((1-z^n(q_1^2q_2+q_1^{-2}q_2^{-1})+z^{2n})(1-z^n(q_1q_2^2+q_1^{-1}q_2^{-2})+z^{2n})\right.\\ \left.\times(1-z^n(q_1q_2^{-1}+q_1^{-1}q_2)+z^{2n})\right)^{a_n(3)}
\end{multline}
for $k\le 7$; if one does calculations in $\mathbb{Q}[q_1,q_2,q_1^{-1},q_2^{-1}][z]/(z^{8})$, these computations are not very complex and are certainly verifiable by a sufficiently thorough human being. It turns out that those coefficients have the residue $0$ for $k\le 6$ as expected from Proposition \ref{prop:lowdeg}, but the coefficient for $k=7$ has the residue equal to $-15$. This implies that the prediction for $\dim\Alt(x_1,x_2,x_3)_k$ is correct for $k\le 6$, but fails for $k=7$. This disproves Conjecture \ref{conj:conj2}, and, due to \cite[Th.~4.1]{shang2025allisonbenkartgaofunctorcyclicityfree}, Conjecture \ref{conj:conj1} as well. 
\end{proof}

\begin{remark}
In fact, using the \texttt{albert} program \cite{10.1145/190347.190358}, we can get more precise information in the case of $\Alt(x_1,x_2,x_3)$. Using that program, we find that the dimension of the multihomogeneous component of multidegree $(3,3,1)$ in $\Alt(x_1,x_2,x_3)$ is equal to $154$, as opposed to the prediction $152$ of the conjecture of \cite{shang2025allisonbenkartgaofunctorcyclicityfree}, and that the dimension of the multihomogeneous component of multidegree $(3,2,2)$ in $\Alt(x_1,x_2,x_3)$ is equal to $236$, as opposed to the prediction $233$ of the conjecture of \cite{shang2025allisonbenkartgaofunctorcyclicityfree}. Altogether these mismatches, multiplied by the cardinalities of $S_3$-orbits of the respective multidegrees, add up to $15$, the total dimension mismatch.
\end{remark}

\subsection{The conjecture and the Schur positivity property}

One may view the calculation of Proposition~\ref{prop:Alt01} as that of the multiplicity of the sign $S_{10}$-module in the component $\Alt(10)$ of the alternative operad. A yet another problem emerges if one computes multiplicities of \emph{all} irreducible $S_{10}$-modules in that component.

\begin{proposition}
Conjecture \ref{conj:conj2} predicts, for the $S_{10}$-action on the component $\Alt(10)$ of the alternative operad, a virtual module that is not effective (that is, includes some irreducible $S_{10}$-modules with negative coefficients).
\end{proposition}

\begin{proof}
Computing the prediction of Conjecture \ref{conj:conj2} using the procedure implemented in Appendix \ref{sec:appendix}, we obtain the following answer in terms of the Schur symmetric functions (which correspond to characters of irreducible $S_n$-modules \cite{MR3443860}):
\begin{multline*}
12s_{1^{10}} + 16s_{2,1^8} + (-12)s_{2^2,1^6} + (-32)s_{2^3,1^4}+ (-27)s_{2^4,1^2} + (-11)s_{2^5}  + (-2)s_{3,1^7}\\ 
 + (-56)s_{3,2,1^5} + (-41)s_{3,2^2,1^3} + 16s_{3,2^3,1} + (-16)s_{3^2,1^4}+ 69s_{3^2,2,1^2} + 67s_{3^2,2^2} \\
  + 68s_{3^3,1} + 26s_{4,1^6} + 95s_{4,2,1^4} + 226s_{4,2^2,1^2} + 145s_{4,2^3} + 251s_{4,3,1^3} + 412s_{4,3,2,1} \\ 
+ 124s_{4,3^2} + 217s_{4^2,1^2}+ 179s_{4^2,2} + 109s_{5,1^5}+ 356s_{5,2,1^3} + 415s_{5,2^2,1}+ 472s_{5,3,1^2} \\
  + 361s_{5,3,2} + 268s_{5,4,1} + 42s_{5^2}  + 151s_{6,1^4} + 365s_{6,2,1^2} + 218s_{6,2^2}+ 307s_{6,3,1} \\
 + 90s_{6,4} + 110s_{7,1^3} + 172s_{7,2,1} + 75s_{7,3} + 44s_{8,1^2} + 35s_{8,2} + 9s_{9,1} + s_{10} .     
\end{multline*}
We note that some of the coefficients are negative: those of 
 \[
s_{2^2,1^6},\quad s_{2^3,1^4},\quad s_{2^4,1^2},\quad s_{2^5},\quad  s_{3,1^7},\quad s_{3,2,1^5},\quad s_{3,2^2,1^3},\quad s_{3^2,1^4}. 
 \]
As a consequence, we see that the predicted answer is a virtual character that is not effective. 
\end{proof}

\begin{remark}
The Schur function among the ones listed in the proof above that corresponds to a partition with the smallest possible number of rows is $s_{2^5}$. Therefore starting from $\Alt(a_1,a_2,a_3,a_4,a_5)$, there are certainly unexpected homology classes in the trivial and/or adjoint component that prevent the conjectural answer even from being Schur-positive (let alone correct).
\end{remark}

\section{Inner derivations of free alternative algebras}\label{sec:inner}

In \cite[Lemma~3.3]{shang2025allisonbenkartgaofunctorcyclicityfree}, it is established that if Conjecture \ref{conj:conj1} is true for a given vector space $V$, then 
 \[
\calB(\Alt(V))=\mathrm{Inner}(\Alt(V)).
 \] 

Let us prove that, even though Conjecture \ref{conj:conj1} fails for the free alternative algebra with two generators, the above conclusion still holds. In fact, we prove a more general statement, from which the statement about the free alternative algebra with two generators follows by virtue of the theorem of Artin.

\begin{proposition}
Let $V$ be a vector space, and consider the tensor algebra $T(V)$ as an alternative algebra. Then 
 \[
\calB(T(V))\cong \mathrm{Inner}(T(V)).     
 \]
\end{proposition}

\begin{proof}
Note that for an associative algebra $A$, we have $[L_a,R_b]=0$ for all $a,b$. Moreover, we have 
 \[
([L_a,L_b]+[R_a,R_b])(c)=abc-bac+cba-cab=[ab-ba,c]=[[a,b],c]. 
 \]
Moreover, since $[T(V),T(V)]\cap Z(T(V))=0$, we have 
 \[
\mathrm{Inner}(T(V))\cong [T(V),T(V)].     
 \]  
At the same time, the fact that the kernel of the canonical surjective map
 \[
\Lambda^2(T(V))\twoheadrightarrow [T(V),T(V)]
 \]
coincides with
 \[
(ab\wedge c+bc\wedge a+ca\wedge b\colon a,b,c\in T(V))    
 \] 
is essentially the computation of the cyclic homology $HC_1(T(V))$, and in any case is well known; see, for instance \cite[Lemma~1.2]{MR2350124}. Thus, we have
 \[
\Lambda^2(T(V))(ab\wedge c+bc\wedge a+ca\wedge b\colon a,b,c\in T(V))\cong [T(V),T(V)],
 \]
and the left hand side of this isomorphism is how $\calB(T(V))$ is defined, which completes the proof.
\end{proof}

This raises a natural question whether we always have 
 \[
\calB(\Alt(V))\cong\mathrm{Inner}(\Alt(V))     
 \]
even though Conjecture \ref{conj:conj1} is not there to help us.

\appendix 

\section{The \texorpdfstring{\texttt{SageMath}}{SageMath} code}\label{sec:appendix}

The following code computes the prediction of Conjecture \ref{conj:conj2} for the symmetric group actions on the components of the operad of alternative algebras. It literally follows the recursive procedure described in \cite[Lemma 4.1]{shang2025allisonbenkartgaofunctorcyclicityfree}. Note that in the language of symmetric functions \cite{MR3443860}, the operation 
 \[
\lambda(x)=\sum_{k\ge 0}(-1)^k[\Lambda^k(U)]     
 \]
used to define the elements $a(V)$ and $b(V)$ corresponds to the plethysm with the element $\sum_{k\ge 0}(-1)^ke_k$. To extract the trivial and adjoint multiplicities we use the Weyl character formula for $\mathfrak{sl}_3$, see, e.g. \cite{MR1153249}.  

\bigskip 

\begin{verbatim}
R.<x,y> = LaurentPolynomialRing(QQ,2)
Sym = SymmetricFunctions(R)
s = Sym.s()
e = Sym.e()

def computealtlazy(dd):
    LazyPowerSeriesRing.options.display_length(dd)
    L = LazySymmetricFunctions(e)
    la = L(lambda n: (-1)^n*e[n])
    a = [0 for t in range (1,dd+2)]
    b = [0 for t in range (1,dd+2)]
    echar = 1
    a[1] = s[1]
    ad = x^2*y+x/y+y/x+x*y^2+1/(x^2*y)+1/(x*y^2)+2
    print(a[1],"\n")
    for i in range (1,dd) :
        echar=echar*la(a[i]*ad+b[i])
        new = (1-x*y^-1)*(1-x*y^2)*(1-x^2*y)*echar
        newp = new.coefficient(i+1)
        a[i+1] = newp.map_coefficients(lambda cf:cf[-2,-1])
        b[i+1] = newp.map_coefficients(lambda cf:cf[0,0])
        print(a[i+1],"\n")
\end{verbatim}

\printbibliography

\end{document}